\documentclass[12pt]{article}
\usepackage{amssymb,amsthm,amsmath,latexsym}
\usepackage{url}
\usepackage{color}
\usepackage{comment}
\newtheorem{thm}{Theorem}[section]
\newtheorem{prop}[thm]{Proposition}
\newtheorem{lem}[thm]{Lemma}

\theoremstyle{remark}
\newtheorem{rem}[thm]{Remark}

\newcommand{\FF}{\mathbb{F}}

\DeclareMathOperator{\supp}{supp}

\begin{document}

\title{Nonexistence of certain singly even self-dual codes
with minimal shadow}

\author{
Stefka Bouyuklieva\thanks{
Faculty of Mathematics and Informatics, 
Veliko Tarnovo University, 
5000 Veliko Tarnovo, Bulgaria.},
Masaaki Harada\thanks{
Research Center for Pure and Applied Mathematics,
Graduate School of Information Sciences,
Tohoku University, Sendai 980--8579, Japan.}
and
Akihiro Munemasa\thanks{
Research Center for Pure and Applied Mathematics,
Graduate School of Information Sciences,
Tohoku University,
Sendai 980--8579, Japan.}
}

\maketitle

\begin{abstract}
It is known that there is no extremal singly even
self-dual $[n,n/2,d]$ code with minimal shadow
for $(n,d)=(24m+2,4m+4)$, $(24m+4,4m+4)$, 
$(24m+6,4m+4)$, $(24m+10,4m+4)$ and $(24m+22,4m+6)$.
In this paper, we study singly even self-dual codes with minimal shadow 
having minimum weight $d-2$ for these $(n,d)$.
For $n=24m+2$, $24m+4$ and $24m+10$, 
we show that the weight enumerator of a singly even self-dual 
$[n,n/2,4m+2]$ code with minimal shadow is uniquely determined
and we also show that there is no 
singly even self-dual $[n,n/2,4m+2]$ code with minimal shadow 
for $m \ge 155$, $m \ge 156$ and $m \ge 160$, respectively.
We demonstrate that the weight enumerator of a singly even self-dual 
code with minimal shadow is not uniquely determined
for parameters $[24m+6,12m+3,4m+2]$ and
$[24m+22,12m+11,4m+4]$.
\end{abstract}

\section{Introduction}

A (binary) code $C$ of length $n$ is a vector subspace of
$\FF_2^n$, where $\FF_2$ denotes the finite field of order $2$.
The {\em dual} code $C^{\perp}$ of $C$ is defined as
$
C^{\perp}=
\{x \in \FF_2^n \mid x \cdot y = 0 \text{ for all } y \in C\},
$
where $x \cdot y$ is the standard inner product.
A code $C$ is called
{\em self-dual} if $C = C^{\perp}$.
Self-dual codes are divided into two classes.
A self-dual code $C$ is {\em doubly even} if all
codewords of $C$ have weight divisible by four, and {\em
singly even} if there is at least one codeword of weight $\equiv 2
\pmod 4$.
Let $C$ be a singly even self-dual code and
let $C_0$ denote the
subcode of codewords having weight $\equiv0\pmod4$.
Then $C_0$ is a subcode of codimension $1$.
The {\em shadow} $S$ of $C$ is defined to be $C_0^\perp \setminus C$.
Shadows for self-dual codes were introduced by Conway and
Sloane~\cite{C-S}
in order to 
derive new upper bounds for the minimum weight of
singly even self-dual codes.
By considering shadows, Rains~\cite{Rains} showed that
the minimum weight $d$ of a self-dual code of length $n$
is bounded by
$d  \le 4 \lfloor{\frac {n}{24}} \rfloor + 6$
if $n \equiv 22 \pmod {24}$,
$d  \le 4  \lfloor{\frac {n}{24}} \rfloor + 4$
otherwise.
A self-dual code meeting the bound is called  {\em extremal}.


Let $C$ be a singly even self-dual code of
length $n$ with shadow $S$.
Let $d(S)$ denote the minimum weight of $S$.
We say that $C$ is a code with {\em minimal shadow} 
if $r = d(S)$, where 
$r=4,1,2$ and $3$ if $n \equiv 0,2,4$ and $6 \pmod 8$,
respectively.
The concept of self-dual codes with minimal shadow was introduced 
in~\cite{performance}. 
In that paper, different types of self-dual codes with the same 
parameters were compared with regard to the decoding error probability. 
In~\cite{BV}, the connection between singly even self-dual 
codes with minimal shadow of some lengths, combinatorial designs and 
secret sharing schemes was considered. 
It was shown in~\cite{BW} that there is no extremal singly even
self-dual code with minimal shadow
for lengths $24m+2$, $24m+4$,
$24m+6$, $24m+10$ and $24m+22$.
In~\cite{BV}, it was shown that the weight enumerator of a 
(non-extremal) singly even self-dual $[24m+2, 12m+1, 4m+2]$ code 
with minimal shadow is uniquely determined
for each positive integer $m$. These motivate us to study singly even
self-dual codes with minimal shadow having minimum weight two less than
the hypothetical extremal case.

%

The main aim of this paper is to investigate singly even 
self-dual codes with minimal shadow having minimum weight 
$4m+2$ for the lengths $24m+2$,
$24m+4$ and $24m+10$. 
We show that the weight enumerator of a singly even self-dual 
code with minimal shadow having minimum weight $4m+2$
is uniquely determined for lengths $24m+4$ and $24m+10$.
For lengths $24m+2$, $24m+4$ and $24m+10$,
nonnegativity of the coefficients of weight enumerators shows 
that there is no such code for $m$ sufficiently large.
We also show that the uniqueness of the weight enumerator
fails for the parameters $[24m+6,12m+3,4m+2]$ and
$[24m+22,12m+11,4m+4]$.

The paper is organized as follows.
In Section~\ref{sec:2}, 
we review the results given by Rains~\cite{Rains}.
In Section~\ref{sec:24m+2}, we show
that there is no singly even self-dual $[24m+2,12m+1,4m+2]$
code with minimal shadow for $m \ge 155$.
In Sections~\ref{sec:24m+4} and~\ref{sec:24m+10}, 
for parameters $[24m+4,12m+2,4m+2]$ and $[24m+10,12m+5,4m+2]$,
we show that there is no 
singly even self-dual code with minimal shadow for $m \ge 156$
and for $m \ge 160$, respectively.
Finally, in Section~\ref{sec:rem}, we 
demonstrate that the weight enumerator of a singly even self-dual 
code with minimal shadow is not uniquely determined
for parameters $[24m+6,12m+3,4m+2]$ and
$[24m+22,12m+11,4m+4]$.


All computer calculations in this paper
were done with the help of 
the algebra software {\sc Magma}~\cite{Magma} and
the mathematical softwares {\sc Maple} and {\sc Mathematica}.

\section{Preliminaries}\label{sec:2}


Let $C$ be a singly even self-dual code of
length $n$ with shadow $S$.
Write $n=24m+8l+2r$, where $m$ is an integer, $l\in\{0,1,2\}$ and
$r\in\{0,1,2,3\}$. 
The weight enumerators $W_C(y)$ and $W_S(y)$ of $C$ and $S$
are given by (\cite[(10), (11)]{C-S})
\begin{align}\label{eq:WC}
W_C(y) &= \sum_{i=0}^{12m+4l+r}a_i y^{2i}
=
\sum_{j=0}^{3m+l}
c_j(1+y^2)^{12m+4l+r-4j}(y^2(1-y^2)^2)^j,
\\
\label{eq:WS}
W_S(y)  &= \sum_{i=0}^{6m+2l}b_i y^{4i+r}
= \sum_{j=0}^{3m+l}(-1)^j c_j2^{12m+4l+r-6j}y^{12m+4l+r-4j}(1-y^4)^{2j},
\end{align}
respectively, for suitable integers $c_j$.
%
Let
\begin{equation}\label{4a8}
(1+y^2)^{n/2-4j}(y^2(1-y^2)^2)^j=
\sum_{i=0}^{12m+4l+r} \alpha'_{i,j} y^{2i}\quad(0\leq j\leq 3m+l).
\end{equation}
Then
\begin{equation}\label{4b0}
\alpha_{i,j}'=\begin{cases}
0&\text{if $0\leq i<j\leq 3m+l$,}\\
1&\text{if $0\leq i=j\leq 3m+l$.}
\end{cases}
\end{equation}
This implies that the $(3m+l+1)\times(3m+l+1)$ matrix $[\alpha'_{i,j}]$
is invertible, since it is unitriangular.
Let $[\alpha_{i,j}]$ be its inverse matrix. 
Then by~\eqref{4b0}, we have
\begin{equation}\label{4b1}
\alpha_{i,j}=\begin{cases}
0&\text{if $0\leq i<j\leq 3m+l$,}\\
1&\text{if $0\leq i=j\leq 3m+l$,}
\end{cases}
\end{equation}
and
\begin{equation}\label{4by}
y^{2i}=\sum_{j=0}^{3m+l}\alpha_{j,i}(1+y^2)^{n/2-4j}(y^2(1-y^2)^2)^j
\quad(0\leq i\leq 3m+l)
\end{equation}
by~\eqref{4a8}.
By~\eqref{eq:WC}, \eqref{4b1} and~\eqref{4by}, we obtain  
\begin{equation}\label{4a10}
c_i=\sum_{j=0}^{i} \alpha_{i,j} a_j.
\end{equation}

\begin{lem}\label{lem:alpha}
For $1\leq i\leq 3m+l$, we have
\begin{equation}\label{eq:alpha}
\alpha_{i,0}=
-\frac{n}{2i}
\sum_{\substack{0\leq t\leq n/2+1-6i\\ 
t+i\text{ is odd}}}
(-1)^t\binom{\frac{n}{2}+1-6i}{t}
\binom{\frac{n-7i-t-1}{2}}{\frac{i-t-1}{2}}.
\end{equation}
\end{lem}
\begin{proof}
For $1 \le i$, 
\begin{align*}
\alpha_{i,0}=&
-\frac{n}{2i}[\text{coeff.\ of } y^{i-1} \text{ in }
(1+y)^{-n/2-1+4i}(1-y)^{-2i}],
\end{align*}
\cite{Rains}.
Since
\begin{align*}
&(1+y)^{-n/2-1+4i}(1-y)^{-2i}\\
&=(1-y^2)^{-n/2-1+4i}(1-y)^{n/2+1-6i} \\
&=(1-y^2)^{-n/2-1+4i}
\sum_{t=0}^{n/2+1-6i}(-1)^t\binom{\frac{n}{2}+1-6i}{t}y^t,
\end{align*}
we have
\begin{align*}
\alpha_{i,0}&=
-\frac{n}{2i}
\sum_{t=0}^{n/2+1-6i}(-1)^t\binom{\frac{n}{2}+1-6i}{t}
[\text{coeff.\ of } y^{i-1} \text{ in }
(1-y^2)^{-n/2-1+4i} y^t]
\\&=
-\frac{n}{2i}
\sum_{\substack{0\leq t\leq n/2+1-6i\\ 
t+i\text{ is odd}}}
(-1)^t\binom{\frac{n}{2}+1-6i}{t}
(-1)^{(i-t-1)/2}\binom{-\frac{n}{2}-1+4i}{\frac{i-t-1}{2}}.
\end{align*}
The result follows by applying the formula 
\[(-1)^j\binom{-n}{j}=\binom{n+j-1}{j}.\]
\end{proof}

Write
\begin{equation*}\label{4c8}
(-1)^j2^{n/2-6j}y^{n/2-4j}(1-y^4)^{2j}=
\sum_{i=0}^{6m+2l} \beta'_{i,j} y^{4i+r}
\quad(0\leq j\leq 3m+l).
\end{equation*}
Since $n/2-4j=4(3m+l-j)+r$, we have
\begin{equation*}\label{4d0}
\beta_{i,j}'=\begin{cases}
0&\text{if $i<3m+l-j$,}\\
(-1)^j 2^{n/2-6j}&\text{if $i=3m+l-j$.}
\end{cases}
\end{equation*}
This implies that the $(3m+l+1)\times(3m+l+1)$ matrix $[\beta'_{i,3m+l-j}]$
is invertible, since it is lower triangular such that
the diagonal elements are not zeros. Thus, the matrix
$[\beta'_{i,j}]$ is also invertible. 
Let $[\beta_{i,j}]$ be its inverse matrix. Then 
\begin{equation}\label{4c9}
y^{4i+r}=\sum_{j=0}^{3m+l} \beta_{j,i}(-1)^j 2^{n/2-6j}y^{n/2-4j}(1-y^4)^{2j}
\quad(0\leq i\leq 3m+l).
\end{equation}
Moreover,
$[\beta_{3m+l-i,j}]$ is the inverse of the lower triangular matrix
$[\beta'_{i,3m+l-j}]$, and so lower triangular as well, and
\[\beta_{3m+l-j,j}={\beta'}^{-1}_{j,3m+l-j}.\]
Thus
\begin{equation}\label{4d1}
\beta_{i,j}=\begin{cases}
0&\text{if $i>3m+l-j$,}\\
(-1)^{3m+l-j} 2^{6(3m+l-j)-n/2}&\text{if $i=3m+l-j$.}
\end{cases}
\end{equation}
By~\eqref{eq:WS}, \eqref{4c9} and~\eqref{4d1}, we obtain
\begin{equation}\label{4c10}
c_i=\sum_{j=0}^{3m+l-i} \beta_{i,j} b_j.
\end{equation}

\begin{lem}[Rains~\cite{Rains}]\label{lem:6b}
For $1\leq i\leq3m+l$ and $0\leq j\leq3m+l$ with $i+j\leq 3m+l$, we have
\begin{equation}\label{eq:beta}
\beta_{i,j}=(-1)^i2^{-n/2+6i}\frac{3m+l-j}{i}\binom{3m+l+i-j-1}{3m+l-i-j}.
\end{equation}
\end{lem}

From~\eqref{4a10} and~\eqref{4c10}, we have
\begin{equation}\label{eq:ci}
c_i=\sum_{j=0}^{i} \alpha_{i,j} a_j=\sum_{j=0}^{3m+l-i} \beta_{i,j} b_j.
\end{equation}

Now let $C$ be a singly even self-dual $[24m+8l+2r,12m+4l+r,4m+2]$ code
with minimal shadow.
Suppose that $(l,r) \in \{(0,1),(0,2),(1,1)\}$.
Since the minimum weight of $C$ is $4m+2$, we have
\begin{equation}\label{eq:ai}
a_0=1,a_1=a_2=\cdots = a_{2m}=0.  
\end{equation}
Since the minimum weight of the shadow is $1$ or $2$, we have
\begin{equation}\label{eq:bi}
\begin{cases}
b_0=1                               & \text{ if } m=1, \\
b_0=1, b_1=b_2= \cdots =b_{m-1}=0   & \text{ if } m \ge 2.
\end{cases}
\end{equation}
From~\eqref{eq:ci}, \eqref{eq:ai} and~\eqref{eq:bi},
we have
\begin{equation}\label{eq:c_i-all}
c_i=
\begin{cases} 
\alpha_{i,0} & \text{ if }i=0,1,\ldots,2m, \\
\beta_{i,0}  & \text{ if }i=2m+l+1,2m+l+2,\ldots,3m+l. 
\end{cases}
\end{equation}
Suppose that $l=0$.
From~\eqref{eq:ci}, \eqref{eq:ai}, 
\eqref{eq:bi} and~\eqref{eq:c_i-all}, we obtain
\begin{align}\label{eq:1}
c_{2m}=&\alpha_{2m,0}
=\beta_{2m,0} + \beta_{2m,m}b_m,
\\
c_{2m-1} =&
\alpha_{2m-1,0}
=\beta_{2m-1,0} + \beta_{2m-1,m}b_m +\beta_{2m-1,m+1}b_{m+1}.
\label{eq:2}
\end{align}
Suppose that $l=1$.
From~\eqref{eq:ci}, \eqref{eq:ai}, 
\eqref{eq:bi} and~\eqref{eq:c_i-all}, we obtain
\begin{align}\label{eq:3}
c_{2m}=&\alpha_{2m,0}=\beta_{2m,0}+\beta_{2m,m}b_{m}+\beta_{2m,m+1}b_{m+1},
\\
c_{2m+1}
=&\alpha_{2m+1,0}+\alpha_{2m+1,2m+1}a_{2m+1}
=\beta_{2m+1,0}+\beta_{2m+1,m}b_m.
\label{eq:4}
\end{align}

\section{Singly even
self-dual $[24m+2,12m+1,4m+2]$ codes with minimal shadow}\label{sec:24m+2}

It was shown in~\cite{BV} that
the weight enumerator of a singly even self-dual $[24m+2,12m+1,4m+2]$
code with minimal shadow is uniquely determined for each length.
In this section, we show that 
there is no singly even self-dual $[24m+2,12m+1,4m+2]$ code
with minimal shadow for $m \ge 155$.

Suppose that $m \ge 1$.
Let $C$ be a singly even self-dual $[24m+2,12m+1,4m+2]$ code
with minimal shadow.
The weight enumerators of $C$ and its shadow $S$ are written 
as in~\eqref{eq:WC} and~\eqref{eq:WS}, respectively.

%
%
From~\eqref{eq:alpha},
\begin{align*}
\alpha_{2m,0}=
\frac{12m+1}{m}\binom{5m}{m-1}.
\end{align*}
From~\eqref{4d1},
\[
\beta_{2m,m}=\frac{1}{2}.
\]
From~\eqref{eq:beta},
\begin{align*}
\beta_{2m,0} 
=\frac{3}{2^2}\binom{5m-1}{m}
=\frac{3(4m+1)}{5m}\binom{5m}{m-1}.
\end{align*}
From~\eqref{eq:1}, 
\begin{align*}
b_m=\frac{\alpha_{2m,0}-\beta_{2m,0}}{\beta_{2m,m}}  
=\frac{4(24m+1)}{5m}\binom{5m}{m-1}.
\end{align*}
\begin{rem}
Unfortunately, $b_m$ was incorrectly  reported
in~\cite{ZMFG}. The correct formula for $b_m$ is given in~\cite{ZMFG2}. 
We showed that $b_m$ is always a positive integer (see~\cite{ZMFG2}).
\end{rem}

%
%

From~\eqref{eq:alpha},
\begin{align*}
\alpha_{2m-1,0}=&
-\frac{12m+1}{2m-1}
\left(
\binom{5m+4}{m-1}
+28\binom{5m+3}{m-2}
+70\binom{5m+2}{m-3}
\right.
\\ &
\left.
+28\binom{5m+1}{m-4}
+\binom{5m}{m-5}
\right)
\\ =&
-
\frac{8(12m+1)(376m^3- 4m^2  + 5m +1)}
{(4m+2)(4m+3)(4m+4)(4m+5)}
\binom{5m}{m-1}.
\end{align*}
From~\eqref{4d1},
\[
\beta_{2m-1,m+1}=-\frac{1}{2^7}.
\]
From~\eqref{eq:beta}, 
\begin{align*}
\beta_{2m-1,0}
=&-\frac{1}{2^{7}}\frac{3m}{2m-1}\binom{5m-2}{m+1}
=-\frac{1}{2^{4}}
\frac{3(4m-1)(4m+1)}{5(5m-1)(m+1)}\binom{5m}{m-1},\\
\beta_{2m-1,m}
=&-\frac{m}{2^{5}}.
\end{align*}
From~\eqref{eq:2}, 
\begin{align*}
b_{m+1} =&\frac{\alpha_{2m-1,0}- \beta_{2m-1,0} - 
\beta_{2m-1,m}b_m}{\beta_{2m-1,m+1}}
%
\\
=&-\frac{64(24m+1)f(m)}
{(5m-1)(4m+2)(4m+3)(4m+4)(4m+5)}\binom{5m}{m-1},
\end{align*}
where 
\[
f(m)=64m^5-14816m^4+2812m^3+46m^2-14m+1.
\]


\begin{thm}
All coefficients in the weight enumerators of 
a singly even self-dual $[24m+2,12m+1,4m+2]$ code
and its shadow are nonnegative integers 
if and only if $1 \le m \le 154$.
In particular,
for $m \ge 155$, there is no 
singly even self-dual $[24m+2,12m+1,4m+2]$
code with minimal shadow.
\end{thm}
\begin{proof}
We verified that the equation 
$f(m)=0$
has three solutions consisting of real numbers and
the largest solution is in the interval $(231,232)$.
Thus, $b_{m+1}$ is negative for $m \ge 232$.
%
Using~\eqref{eq:WC} and~\eqref{eq:WS},
we determined numerically 
the weight enumerators of 
a singly even self-dual $[24m+2,12m+1,4m+2]$
code with minimal shadow and its shadow
for $m \le 231$.
The theorem follows from this calculation.
\end{proof}

\section{Singly even
self-dual $[24m+4,12m+2,4m+2]$ codes with minimal shadow}\label{sec:24m+4}

\begin{prop}
The weight enumerator of a singly even self-dual $[24m+4,12m+2,4m+2]$
code with minimal shadow is uniquely determined for each length.
\end{prop}
\begin{proof}
The weight enumerator of a singly even self-dual 
$[4,2,2]$ code is uniquely determined.
Suppose that $m \ge 1$.
Let $C$ be a singly even self-dual $[24m+4,12m+2,4m+2]$ code
with minimal shadow.
The weight enumerators of $C$ and its shadow $S$ are written 
as using~\eqref{eq:WC} and~\eqref{eq:WS}.
Since $\alpha_{i,0}$ $(i=0,1,\ldots,2m)$ and
$\beta_{i,0}$ $(i=2m+1,2m+2,\ldots,3m)$ are calculated by
\eqref{eq:alpha} and~\eqref{eq:beta}, respectively,
from~\eqref{eq:c_i-all}, 
$c_i$ $(i=0,1,\ldots,3m)$ depends only on $m$.
This means that the weight enumerator of $C$ is uniquely
determined for each length.
\end{proof}

From~\eqref{eq:alpha},
\begin{align*}
\alpha_{2m,0}=&
\frac{6m+1}{m}\left(3\binom{5m+1}{m-1}+\binom{5m}{m-2}\right)
\\=&
\frac{(6m+1)(8m+1)}{m(2m+1)}\binom{5m}{m-1}.
\end{align*}
From~\eqref{4d1},
\[
\beta_{2m,m}=\frac{1}{2^2}.   
\]
From~\eqref{eq:beta},
\begin{align*}
\beta_{2m,0}
=\frac{1}{2^2}\frac{3}{2}\binom{5m-1}{m}
=\frac{3(4m+1)}{10m}\binom{5m}{m-1}.
\end{align*}
From~\eqref{eq:1}, 
\begin{align*}
b_m=&\frac{\alpha_{2m,0}-\beta_{2m,0}}{\beta_{2m,m}}  
=
\frac{2(12m+1)(38m+7)}{5m(2m+1)}\binom{5m}{m-1}.
\end{align*}
\begin{rem}
Unfortunately, $b_m$ was incorrectly  reported
in~\cite{ZMFG}. The correct formula for $b_m$ is given in~\cite{ZMFG2}. 
We showed that $b_m$ is always a positive integer (see~\cite{ZMFG2}).
\end{rem}

From~\eqref{eq:alpha}, 
\begin{align*}
\alpha_{2m-1,0}=&
-\frac{12m+2}{2m-1}
\left(
\binom{5m+5}{m-1}
+36\binom{5m+4}{m-2}
+126\binom{5m+3}{m-3}\right.
\\ &
\left.
+84\binom{5m+2}{m-4}
+9\binom{5m+1}{m-5}
\right)
\\=&
-\frac{16 (5m+1)(6m+1)(8m+1)(68m^2-m+3)}
{(4m+2)(4m+3)(4m+4)(4m+5)(4m+6)}\binom{5m}{m-1}.
\end{align*}
From~\eqref{4d1},
\[
\beta_{2m-1,m+1}=-\frac{1}{2^8}.  
\]
From~\eqref{eq:beta}, 
\begin{align*}
\beta_{2m-1,0}
=&-\frac{1}{2^{8}}\frac{3m}{2m-1}\binom{5m-2}{m+1} 
=-\frac{1}{2^{5}}
\frac{3(4m-1)(4m+1)}{5(5m-1)(m+1)}\binom{5m}{m-1}, \\
\beta_{2m-1,m}
=&-\frac{m}{2^6}.
\end{align*}
From~\eqref{eq:2}, 
\begin{align*}
b_{m+1} =&\frac{\alpha_{2m-1,0}- \beta_{2m-1,0} - 
\beta_{2m-1,m}b_m}{\beta_{2m-1,m+1}}
%
\\ 
=&
-\frac{128(12m+1)f(m)}
{(5m-1)(4m+2)(4m+3)(4m+4)(4m+5)(4m+6)}
\binom{5m}{m-1},
\end{align*}
where
\[
f(m) =
1216 m^6-212096 m^5-33020 m^4+5440 m^3+1171 m^2+88 m+6.
\]


\begin{thm}
All coefficients in 
the weight enumerators of 
a singly even self-dual $[24m+4,12m+2,4m+2]$ code
and its shadow are nonnegative integers 
if and only if $1 \le m \le 155$.
In particular, for $m \ge 156$, there is no 
singly even self-dual $[24m+4,12m+2,4m+2]$
code with minimal shadow.
\end{thm}
\begin{proof}
We verified that the equation 
$f(m)=0$
has two solutions consisting of real numbers and
the largest solution is in the interval $(174,175)$.
Thus, $b_{m+1}$ is negative for $m \ge 175$.
%
Using~\eqref{eq:WC} and~\eqref{eq:WS},
we determined numerically 
the weight enumerators of 
a singly even self-dual $[24m+4,12m+2,4m+2]$
code with minimal shadow and its shadow
for $m \le 174$.
The theorem follows from this calculation.
\end{proof}

\section{Singly even
self-dual $[24m+10,12m+5,4m+2]$ codes with minimal shadow}\label{sec:24m+10}

\begin{lem}[Harada~\cite{H60}]
\label{lem:H}
Suppose that $n \equiv 2 \pmod 8$.
Let $C$ be a singly even self-dual $[n,n/2,d]$ code
with minimal shadow.
If $d \equiv 2 \pmod 4$, then $a_{d/2} =b_{(d-2)/4}$.
\end{lem}

As a consequence, the weight enumerator of a singly even
self-dual $[58,29,10]$ code with minimal shadow
was uniquely determined in~\cite{H60}.  

\begin{prop}
The weight enumerator of a singly even self-dual 
$[24m+10,12m+5,4m+2]$ code with minimal shadow
is uniquely determined for each length.
\end{prop}
\begin{proof}
The weight enumerator of a singly even self-dual 
$[10,5,2]$ code with minimal shadow is uniquely determined.
Suppose that $m \ge 1$.
Let $C$ be a singly even self-dual $[24m+10,12m+5,4m+2]$ code 
with minimal shadow.
The weight enumerators of $C$ and its shadow $S$ are written 
as in~\eqref{eq:WC} and~\eqref{eq:WS}, respectively.
Since $\alpha_{i,0}$ $(i=0,1,\ldots,2m)$ and
$\beta_{i,0}$ $(i=2m+2,2m+3,\ldots,3m+1)$ are calculated 
by~\eqref{eq:alpha} and~\eqref{eq:beta}, respectively,
from~\eqref{eq:c_i-all}, 
$c_i$ $(i=0,1,\ldots,2m,2m+2,\ldots,3m+1)$ depends only on $m$.

From~\eqref{4b1} and~\eqref{4d1}, we have
\[
\alpha_{2m+1,2m+1}=1 \text{ and }
\beta_{2m+1,m}=-2,
\]
respectively.
By Lemma~\ref{lem:H}, it holds that $a_{2m+1}=b_m$.
From~\eqref{eq:4}, we obtain
\begin{equation}\label{eq:a2m+1}
a_{2m+1}=\frac{\beta_{2m+1,0}-\alpha_{2m+1,0}}{3}.
\end{equation}
Therefore, from~\eqref{eq:4}, 
$c_{2m+1}$ depends only on $m$.
This means that the weight enumerator of $C$ is uniquely
determined for each length.
\end{proof}

From~\eqref{eq:alpha},
we have
\begin{align*}
\alpha_{2m+1,0}=&
-\frac{12m+5}{2m+1}\binom{5m+1}{m}.
\end{align*}
From~\eqref{eq:beta}, we have
\begin{align*}
\beta_{2m+1,0}
= -2 \frac{3m+1}{2m+1}\binom{5m+1}{m}.
\end{align*}
Since $a_{2m+1}=b_m$, 
from~\eqref{eq:a2m+1}, we have 
\[
b_m=\binom{5m+1}{m}=\frac{5m+1}{4m+1}\binom{5m}{m}.  
\]
%
%
%
%
%
%
From~\eqref{eq:alpha},
\begin{align*}
\alpha_{2m,0}=&
\frac{12m+5}{2m}\left(
6\binom{5m+4}{m-1}
+20\binom{5m+3}{m-2}
+6\binom{5m+2}{m-3}
\right)
\\
=&
\frac{4(12m+5)(5m+1)(5m+2)(32m^2 +19m+3)}{
(4m+1)(4m+2)(4m+3)(4m+4)(4m+5)}\binom{5m}{m}.
\end{align*}
From~\eqref{4d1},
\[
\beta_{2m,m+1}=\frac{1}{2^5}.   
\]
From~\eqref{eq:beta},
\begin{align*}
\beta_{2m,0}
=&\frac{1}{2^5}\frac{3m+1}{2m} \binom{5m}{m+1}
=\frac{1}{2^4}\frac{3m+1}{m+1} \binom{5m}{m},
\\
\beta_{2m,m}
=&\frac{1}{2^5} \frac{2m+1}{2m} 4m
=\frac{2m+1}{2^4}.
\end{align*}
From~\eqref{eq:3},
\begin{align*}
b_{m+1}=&\frac{\alpha_{2m,0}-\beta_{2m,0}-\beta_{2m,m}b_{m}}{\beta_{2m,m+1}}\\
=&
-\frac{16(5m+2)f(m)}
{(4m+1)(4m+2)(4m+3)(4m+4)(4m+5)}
\binom{5m}{m},
\end{align*}
where
\[
 f(m)=
64m^5- 15040m^4- 18036m^3 - 7924m^2 - 1511m -105.
\]

\begin{thm}
All coefficients in 
the weight enumerators of 
a singly even self-dual $[24m+10,12m+5,4m+2]$ code
and its shadow are nonnegative integers 
if and only if $1 \le m \le 159$.
In particular, for $m \ge 160$, there is no 
singly even self-dual $[24m+10,12m+5,4m+2]$
code with minimal shadow.
\end{thm}
\begin{proof}
We verified that the equation 
$f(m)=0$
has three solutions consisting of real numbers and
the largest solution is in the interval $(236,237)$.
Thus, $b_{m+1}$ is negative for $m \ge 237$.
%
Using~\eqref{eq:WC} and~\eqref{eq:WS},
we determined numerically 
the weight enumerators of 
a singly even self-dual $[24m+10,12m+5,4m+2]$
code with minimal shadow and its shadow
for $m \le 236$.
The theorem follows from this calculation. 
\end{proof}


\section{Remaining cases}\label{sec:rem}

For the remaining cases, we demonstrate that
the weight enumerator of a singly even self-dual 
code with minimal shadow is not uniquely determined.

\subsection{Singly even
self-dual $[24m+6,12m+3,4m+2]$ codes with minimal shadow}

Using~\eqref{eq:WC} and~\eqref{eq:WS},
the possible weight enumerators of 
a singly even self-dual $[30,15,6]$ code with minimal shadow
and its shadow are given by
\begin{align*}
&
1+( 35  - 8 \beta) y^6 
+( 345 + 24 \beta) y^8 
+ 1848 y^{10} 
+ \cdots,
\\&
\beta y^3 
+( 240 - 6 \beta) y^7 
+( 6720 + 15 \beta) y^{11} 
+ \cdots,
\end{align*}
respectively, where $\beta$ is an integer with $1 \le \beta \le 4$.
It is known that there is 
a singly even self-dual $[30,15,6]$ code with minimal shadow
for $\beta \in \{1,2\}$ (see~\cite{C-S}).

Using~\eqref{eq:WC} and~\eqref{eq:WS},
the possible weight enumerators of 
a singly even self-dual $[54,27,10]$ code with minimal shadow
and its shadow are given by
\begin{align*}
&
1 
+( 351  - 8 \beta) y^{10} 
+( 5543  + 24 \beta) y^{12} 
+( 43884  + 32 \beta) y^{14} 
+ \cdots,
\\&
y^3 
+(- 12 + \beta) y^{7 }
+( 2874 - 10 \beta) y^{11} 
+( 258404 + 45 \beta) y^{15} 
+ \cdots,
\end{align*}
respectively, where $\beta$ is an integer with $12 \le \beta \le 43$.
It is known that there is 
a singly even self-dual $[54,27,10]$ code with minimal shadow
for $\beta \in \{12,13,\ldots,20,21,22,24,26\}$ (see~\cite{YL14}).

\subsection{Singly even
self-dual $[24m+22,12m+11,4m+4]$ codes with minimal shadow}

Using~\eqref{eq:WC} and~\eqref{eq:WS},
the possible weight enumerators of 
a singly even self-dual $[22,11,4]$ code with minimal shadow
and its shadow are given by
\begin{align*}
&
1 
+ 2 \beta y^4 
+( 77  - 2 \beta) y^6 
+( 330  - 6 \beta) y^8 
+( 616  + 6 \beta) y^{10} 
+ \cdots,
\\&
\beta y^3 
+( 352  - 4 \beta) y^7 
+( 1344  + 6 \beta) y^{11} 
+ \cdots,
\end{align*}
respectively, where $\beta$ is an integer with $1 \le \beta \le 38$.
It is known that there is a 
singly even self-dual $[22,11,4]$ code with minimal shadow
for $\beta \in \{2,4,6,8,10,14\}$ (see~\cite{PS75}).
 
Using~\eqref{eq:WC} and~\eqref{eq:WS},
the possible weight enumerators of 
a singly even self-dual $[46,23,8]$ code with minimal shadow
and its shadow are given by
\begin{align*}
&
1 
+ 2 \beta y^8 
+ (884 - 2 \beta) y^{10} 
+ (10556 - 14 \beta) y^{12} 
+ (54621  + 14 \beta) y^{14} 
+ \cdots,
\\&
y^3 
+(- 10  + \beta) y^7 
+( 6669  - 8 \beta) y^{11} 
+( 242760  + 28 \beta) y^{15} 
+ \cdots,
\end{align*}
respectively, where $\beta$ is an integer with $10 \le \beta \le 442$.
Let $C_{46}$ be the code with generator matrix 
$\left[
\begin{array}{cc}
I_{23} & R
\end{array}
\right]$,
where $I_{23}$ denotes the identity matrix of order $23$ and
$R$ is the $23 \times 23$ circulant matrix with
first row
\[
(01011101011100000111110). 
\]
We verified that $C_{46}$ is a singly even self-dual
$[46,23,8]$ code.
By considering self-dual neighbors of $C_{46}$,
we found  singly even self-dual
$[46,23,8]$ codes $N_{46,i}$ with minimal shadow $(i=1,2,\ldots,10)$.
These codes are constructed as
$\langle C_{46} \cap \langle x \rangle^\perp, x \rangle$,
where the supports $\supp(x)$ of $x$ are listed in Table~\ref{Tab:nei}.
The values $\beta$ in the weight enumerators of $N_{46,i}$
are also listed in the table.

\begin{table}[thb]
\caption{Singly even self-dual $[46,23,8]$ codes $N_{46,i}$ with 
minimal shadows}
\label{Tab:nei}
\begin{center}
{\small
\begin{tabular}{c|l|c}
\noalign{\hrule height0.8pt}
Codes & \multicolumn{1}{c|}{$\supp(x)$} & $\beta$\\
\hline
$N_{46,1}$ & $\{1,24,26,27,29,30,31,32,33,34,36,37,42,43,45,46\}$&36\\
$N_{46,2}$ & $\{1,27,28,31,33,35,36,37,42,43,45,46\}$&42\\
$N_{46,3}$ & $\{10,11,20,27,29,34,38,41,42,45\}$&44\\
$N_{46,4}$ & $\{5,6,25,29,30,32,33,36,40,41,44,45\}$&46\\
$N_{46,5}$ & $\{1,23,28,29,30,31,32,37,40,41,44,45\}$&48\\
$N_{46,6}$ & $\{1,26,27,28,30,32,35,36,37,42,43,45\}$&50\\
$N_{46,7}$ & $\{2,3,24,25,26,28,29,33,34,36,37,41,42,44\}$&52\\
$N_{46,8}$ & $\{1,25,28,29,32,33,34,36,38,42,43,45\}$&54\\
$N_{46,9}$ & $\{1,23,24,27,30,36,40,41,44,45\}$&56\\
$N_{46,10}$& $\{1,2,25,29,30,33,35,38,44,46\}$&58\\
\noalign{\hrule height0.8pt}
\end{tabular}
}
\end{center}
\end{table}

The possible weight enumerators of 
a singly even self-dual $[70,35,12]$ code with minimal shadow
and its shadow are given by
\begin{align*}
&
1 
+ 2 \beta y^{12} 
+( 9682 - 2 \beta) y^{14} 
+( 173063 - 22 \beta) y^{16} 
+ \cdots,
\\&
y^3 
+(- 104 + \beta )y^{11}
+( 88480 - 12 \beta) y^{15}
+ \cdots,
\end{align*}
respectively, where $\beta$ is an integer 
$104 \le \beta \le 4841$~\cite{H70}.
It is known that there is a 
singly even self-dual $[70,35,12]$ code with minimal shadow
for many different $\beta$~\cite[p.~1191]{YLGI}.

\bigskip
\noindent
{\bf Acknowledgment.}
The first author 
is supported by Grant DN 02/2/13.12.2016 of the 
Bulgarian National Science Fund.
The second author 
is supported by JSPS KAKENHI Grant Number 15H03633.


\end{document}